\documentclass[12pt,reqno]{amsart}
\usepackage{latexsym}
\usepackage[leqno]{amsmath}
\usepackage{amssymb}
\usepackage{verbatim}
\usepackage{version}
\usepackage{amscd}
\usepackage{bbold}
\usepackage{epsfig}
\catcode`\@=11
\@namedef{subjclassname@2010}{%
  \textup{2010} Mathematics Subject Classification}
\catcode`\@=12

\newtheorem{theorem}{Theorem}
\newtheorem{proposition}[theorem]{Proposition}

\newtheorem{lemma}[theorem]{Lemma}
\newtheorem{corollary}[theorem]{Corollary}

\newtheorem{remarks}[theorem]{Remarks}

\newtheorem{remark}[theorem]{Remark}

\numberwithin{equation}{section}

\def\lcm{\operatorname{lcm}}

\newcommand{\Z}{\mathbb{Z}}

\newcommand{\R}{\mathbb{\R}}

\includeversion{comment}
\begin{document}

\title[Normalising graphs of groups]{Normalising graphs of groups} 
\author[C. Krattenthaler and T.\,W. M\"uller]{C. Krattenthaler$^\dagger$ and Thomas M\"uller$^*$} 

\address{$^{\dagger}$Fakult\"at f\"ur Mathematik, Universit\"at Wien, 
Oskar-Morgenstern-Platz~1, A-1090 Vienna, Austria.
WWW: {\tt http://www.mat.univie.ac.at/\lower0.5ex\hbox{\~{}}kratt}.}

\address{$^*$School of Mathematical Sciences, Queen Mary
\& Westfield College, University of London,
Mile End Road, London E1 4NS, United Kingdom.}

\thanks{$^\dagger$Research partially supported by the Austrian
Science Foundation FWF, grant S50-N15,
in the framework of the Special Research Program
``Algorithmic and Enumerative Combinatorics"\newline\indent
$^*$Research supported by Lise Meitner Grant M1661-N25 of the Austrian
Science Foundation FWF}



\begin{abstract}
We discuss a partial 
normalisation of a finite graph of finite groups $(\Gamma(-), X)$
which leaves invariant the fundamental group. In conjunction with an
easy graph-theoretic result, this provides a flexible and rather
useful tool in the study of finitely generated virtually free
groups. Applications discussed here include (i) an important
inequality for the number of edges in a Stallings decomposition
$\Gamma \cong \pi_1(\Gamma(-), X)$ of a finitely generated virtually
free group, (ii) the proof of  equivalence of a number of conditions
for such a group to be `large', as well as (iii) the classification up
to isomorphism of virtually free groups of (free) rank $2$. We also
discuss some number-theoretic consequences of the last result.  
\end{abstract}

\maketitle

\section{Introduction}
\label{sec:intro}
 The purpose of this paper is to introduce, and demonstrate the usefulness of, a technique for partially normalising the presentation of a finitely generated virtually free group as the fundamental group of a finite graph of finite groups. Roughly speaking, our method avoids trivial amalgamations along a maximal tree of the connected graph underlying such a representation. This result, Lemma~\ref{Lem:Normalise}, in conjunction with an almost trivial graph-theoretic result (Lemma~\ref{Lem:TreeOrient}), provides us with a flexible and  rather powerful tool in the study of such groups. We demonstrate the usefulness of our approach by describing various  applications: (i) a short and elegant argument establishing the (well-known) classification of virtually infinite-cyclic groups due originally to Stallings and Wall, (ii) the classification of virtually free groups of free rank $2$ together with some number-theoretic consequences,\footnote{See Section~\ref{Sec:Prelim} for the definition of the free rank.} and (iii) the equivalence of a number of conditions on a finitely generated virtually free group $\Gamma$ expressing, in one way or other, the fact that $\Gamma$ is large; 
cf.\ Propositions~\ref{Prop:InfCyc}, \ref{Prop:mu=2}, and \ref{Prop:chi<0Equivs}. In Section~\ref{Sec:Ineq}, we also show  that, (iv) for a normalised decomposition $(\Gamma(-), X)$ of a finitely generated virtually free group $\Gamma$, the number of geometric edges
of the graph $X$ is bounded above by the free rank of $\Gamma$; 
cf.\ Lemma~\ref{Lem:EEst}. This important observation plays a role in the proof of Proposition~\ref{Prop:mu=2} below, as well as in establishing certain finiteness results for the class of finitely generated virtually free groups with specified information concerning the number of free subgroups of finite index. 
For another, recent application of the normalisation provided by
Lemma~\ref{Lem:Normalise} see~\cite{KrMuAJ}.

\section{Some preliminaries on finitely generated virtually free groups}
\label{Sec:Prelim}

 Our  notation and terminology here follows Serre's book
\cite{Serre2}; 
in particular, the category of graphs used is described in
\cite[\S2]{Serre2}. 
This category deviates slightly from the usual notions in graph theory.
Specifically, a {\it graph} $X$ consists of two sets:
$E(X)$, the set of (directed) {\it edges}, and $V(X)$, the set of
{\it vertices}. The set $E(X)$ is endowed with a fixed-point-free involution
${}^-: E(X) \rightarrow E(X)$ ({\it reversal of orientation}), and there are
two functions $o,t: E(X)\rightarrow V(X)$ assigning to an edge $e\in
E(X)$ its {\it origin} $o(e)$ and {\it terminus} $t(e)$, such that
$t(\bar{e}) = o(e)$. 
The reader should note that, according to the above
definition, graphs may have loops (that is, edges $e$ with
$o(e)=t(e)$) and multiple edges (that is, several edges with
the same origin and the same terminus).  
An {\it orientation} $\mathcal{O}(X)$ consists of a choice of exactly 
one edge in each pair $\{e, \bar{e}\}$ 
(this is indeed always a {\it pair} -- even for loops --
since, by definition, the involution
${}^-$ is fixed-point-free). Such a pair is called a
{\it geometric edge}. 

 Let $\Gamma$ be a finitely generated virtually free group with
Stallings decomposition\break
 $(\Gamma(-), X)$; that is, $(\Gamma(-), X)$ is
a finite graph of finite groups with fundamental group 
$\pi_1(\Gamma(-), X) \cong \Gamma$. 
If $\mathfrak{F}$ is a free subgroup of finite
index in $\Gamma$ then, following an idea of C.\,T.\,C. Wall, one
defines the (rational) Euler characteristic $\chi(\Gamma)$ of $\Gamma$
as 
\begin{equation}
\label{Eq:EulerChar}
\chi(\Gamma) = - \frac{\mathrm{rk}(\mathfrak{F}) - 1}{(\Gamma:\mathfrak{F})}.
\end{equation}
(This is well-defined in view of Schreier's index formula  in \cite{Schreier}.)  In terms of the above decomposition of $\Gamma$, we have
\begin{equation}
\label{Eq:EulerCharDecomp}
\chi(\Gamma) = \sum_{v\in V(X)} \frac{1}{\vert
\Gamma(v)\vert}\,-\,\sum_{e\in\mathcal{O}(X)} \frac{1}{\vert
\Gamma(e)\vert}. 
\end{equation}
Equation~\eqref{Eq:EulerCharDecomp} reflects the fact that, in our
situation, the Euler characteristic in the sense of Wall coincides
with the equivariant Euler characteristic $\chi_T(\Gamma)$ of $\Gamma$
relative to the tree $T$ canonically associated with $\Gamma$ in the
sense of Bass--Serre theory; cf.\ \cite[Chap.~IX, Prop.~7.3]{Brown} or
\cite[Prop.~14]{Serre1}. We remark that a  finitely generated
virtually free group $\Gamma$ is largest among finitely generated
groups in the sense of Pride's preorder \cite{Pride} (i.e., $\Gamma$
has a subgroup of finite index, which can be mapped onto the free
group of rank $2$) if, and only if, $\chi(\Gamma)<0$; see
Proposition~\ref{Prop:chi<0Equivs} in Section~\ref{Sec:Large}.   

 Denote by $m_\Gamma$ the least common multiple
of the orders of the finite subgroups in $\Gamma$, so that, again in
terms of the above Stallings decomposition of $\Gamma$,  
\[
m_\Gamma = \lcm\big\{\vert\Gamma(v)\vert:\, v\in V(X)\big\}.
\]
(This formula essentially follows from the well-known fact that a
finite group has a fixed point when acting on a tree.)  
The type $\tau(\Gamma)$ of a finitely generated virtually free group
$\Gamma \cong \pi_1(\Gamma(-), X)$ is defined as the  
tuple
\[
\tau(\Gamma) = \big(m_\Gamma; \zeta_1(\Gamma), \ldots, \zeta_\kappa(\Gamma), \ldots, \zeta_{m_\Gamma}(\Gamma)\big),
\]
where the  
$\zeta_\kappa(\Gamma)$'s are integers indexed by the divisors of
$m_\Gamma$, given by 
\[
\zeta_\kappa(\Gamma) = 
\big\vert\big\{e\in \mathcal{O}(X):\, \vert\Gamma(e)\vert
\,\big\vert\, \kappa\big\}\big\vert\,-\, \big\vert\big\{v\in V(X):\,
\vert\Gamma(v)\vert \,\big\vert\, \kappa\big\}\big\vert. 
\]
It can be shown that the type $\tau(\Gamma)$ is in fact an invariant
of the group $\Gamma$, i.e., independent of the particular
decomposition of $\Gamma$ in terms of a graph of groups $(\Gamma(-),
X)$, and that two finitely generated virtually free groups $\Gamma_1$
and $\Gamma_2$ contain the same number of free subgroups of index $n$
for each positive integer $n$ if, and only if, $\tau(\Gamma_1) =
\tau(\Gamma_2)$; cf.\ \cite[Theorem~2]{MuDiss}. We have
$\zeta_\kappa(\Gamma)\geq0$ for $\kappa<m_\Gamma$ and
$\zeta_{m_\Gamma}(\Gamma)\geq-1$ with equality occurring in the latter
inequality if, and only if, $\Gamma$ is the fundamental group of a
tree of groups; cf.\ \cite[Prop.~1]{MuEJC} or \cite[Lemma~2]{MuDiss}. 
We observe that, as a
consequence of (\ref{Eq:EulerCharDecomp}), the Euler characteristic of
$\Gamma$ can be expressed in terms of the type $\tau(\Gamma)$ via 
\begin{equation}
\label{Eq:EulerTypeRewrite}
\chi(\Gamma) = - m_\Gamma^{-1} \sum_{\kappa \mid m_\Gamma}
\varphi(m_\Gamma/\kappa) \,\zeta_\kappa(\Gamma), 
\end{equation}
where $\varphi$ is Euler's totient function. It follows in particular 
that, if two finitely generated virtually free groups have the same
number of free subgroups of index $n$ for every $n$, then their Euler
characteristics must coincide. 

 Define a
\textit{torsion-free $\Gamma$-action} on a set $\Omega$ to be a
$\Gamma$-action on $\Omega$ which is free when restricted to finite
subgroups, and let 
\begin{equation} \label{eq:gla} 
g_\lambda(\Gamma):= \frac{\mbox{number of torsion-free
$\Gamma$-actions on a set with $\lambda m_\Gamma$ elements}}{(\lambda
m_\Gamma)!},\quad \lambda\geq0; 
\end{equation}
in particular, $g_0(\Gamma)=1$. The sequences
$\big(f_\lambda(\Gamma)\big)_{\lambda\geq1}$ and
$\big(g_\lambda(\Gamma)\big)_{\lambda\geq0}$ are related via the
Hall-type convolution formula\footnote{See \cite[Cor.~1]{MuDiss}, or
\cite[Prop.~1]{DM} for a more general result.} 
\begin{equation}
\label{Eq:Transform}
\sum_{\mu=0}^{\lambda-1} g_\mu(\Gamma) f_{\lambda-\mu}(\Gamma) =
m_\Gamma \lambda g_\lambda(\Gamma),\quad \lambda\geq1. 
\end{equation}
Introducing the generating functions 
\[
F_\Gamma(z) := \sum_{\lambda\geq0} f_{\lambda+1}(\Gamma) z^\lambda
\,\mbox{ and }\, G_\Gamma(z):= \sum_{\lambda\geq0} g_\lambda(\Gamma)
z^\lambda, 
\]
Equation~\eqref{Eq:Transform} is seen to be equivalent to the relation
\begin{equation}
\label{Eq:GFTransform}
F_\Gamma(z) = m_\Gamma \frac{d}{dz}\big(\log G_\Gamma(z)\big).
\end{equation}
Moreover, a careful analysis of the universal mapping property
associated with the presentation $\Gamma\cong\pi_1(\Gamma(-), X)$
leads to the explicit formula 
\begin{equation}
\label{Eq:gExplicit}
g_\lambda(\Gamma) = \frac{\prod\limits_{e\in\mathcal{O}(X)} (\lambda
m_\Gamma/\vert\Gamma(e)\vert)!\,\vert\Gamma(e)\vert^{\lambda
m_\Gamma/\vert\Gamma(e)\vert}}{\prod\limits_{v\in V(X)}(\lambda
m_\Gamma/\vert\Gamma(v)\vert)!\,\vert\Gamma(v)\vert^{\lambda
m_\Gamma/\vert\Gamma(v)\vert}},\quad \lambda\geq0 ,
\end{equation}
for $g_\lambda(\Gamma)$, where $\mathcal{O}(X)$ is any orientation of
$X$; cf.\ \cite[Prop.~3]{MuDiss}. 

 Define the \textit{free rank} $\mu(\Gamma)$ of a finitely generated
virtually free group $\Gamma$ to be the rank of a free subgroup of
index $m_\Gamma$ in $\Gamma$ (existence of such a subgroup follows,
for instance, from Lemmas~8 and 10 in \cite{Serre2}; it need not be 
unique, though). We note that, in view of \eqref{Eq:EulerChar}, the
quantity $\mu(\Gamma)$ is connected with the Euler characteristic of
$\Gamma$ via  
\begin{equation}
\label{Eq:FreeEuler}
\mu(\Gamma) + m_\Gamma \chi(\Gamma) = 1,
\end{equation}
which shows in particular that $\mu(\Gamma)$ is
well-defined. 
From Formula~(\ref{Eq:gExplicit}) it may be deduced that the sequence
$g_\lambda(\Gamma)$ is of hypergeometric type and that its generating
function $G_\Gamma(z)$ satisfies a homogeneous linear differential
equation 
\begin{equation}
\label{Eq:G(z)Diff}
\theta_0(\Gamma) G_\Gamma(z) \,+ \,(\theta_1(\Gamma) z - m_\Gamma)
G'_\Gamma(z) \,+ \,\sum _{\mu=2}^{\mu(\Gamma)} \theta_\mu(\Gamma)
z^\mu G^{(\mu)}_\Gamma(z) = 0 
\end{equation}
of order $\mu(\Gamma)$ with integral coefficients $\theta_\mu(\Gamma)$ given by
\begin{equation}
\label{Eq:G(z)DiffCoeffs}
\theta_\mu(\Gamma) = \frac{1}{\mu!} \sum_{j=0}^\mu (-1)^{\mu-j}
\binom{\mu}{j} m_\Gamma (j+1) \prod_{\kappa\mid
  m_\Gamma}\,\underset{(m_\Gamma, k) = \kappa}{\prod_{1\leq k\leq
    m_\Gamma}} (jm_\Gamma + k)^{\zeta_\kappa(\Gamma)},\quad 0\leq
\mu\leq \mu(\Gamma); 
\end{equation}
cf.\ \cite[Prop.~5]{MuDiss}.

\section{Normalising a finite graph of groups}
\label{Sec:GoGNorm}
 It will be important to be able to represent a finitely generated
virtually free group $\Gamma$ by a graph of groups avoiding trivial
amalgamations along a maximal tree. This is achieved via the following. 

\begin{lemma}[\sc Normalisation]
\label{Lem:Normalise}
Let $(\Gamma(-), X)$ be a {\em(}connected{\em)} graph of groups with
fundamental 
group $\Gamma,$ and suppose that $X$ has only finitely many
vertices. Then there exists a graph of groups $(\Delta(-), Y)$ with
$\vert V(Y)\vert < \infty$ and a spanning tree $T$ in $Y,$ such that
$\pi_1(\Delta(-),Y) \cong \Gamma$, and such that\footnote{The notation 
used in Equation~\eqref{eq:normal} follows Serre; 
see D\'ef.~8 in \cite[Sec.~4.4]{Serre2}.} 
\begin{equation} \label{eq:normal}
\Delta(e)^e \neq \Delta(t(e))\,\mbox{ and }\, \Delta(e)^{\bar{e}} \neq
\Delta(o(e)),\quad 
\text {for }e\in E(T). 
\end{equation}
Moreover, if $(\Gamma(-), X)$ satisfies the finiteness condition
\bigskip

\centerline{\hskip3cm
$(F_1)$ \qquad \mbox{$X$ is a finite graph,}\hfill}
\bigskip

or
\bigskip

\centerline{\hskip3cm $(F_2)$  \qquad 
\mbox{$\Gamma(v)$ is finite for every vertex $v\in V(X),$}\hfill}
\bigskip

then we may choose $(\Delta(-), Y)$ so as to enjoy the same property.
\end{lemma}
\begin{proof}
Choose a 
spanning tree $S$ in $X$, and call an edge $e\in E(S)$
\textit{trivial}, if at least one of the associated embeddings $e:
\Gamma(e) \rightarrow \Gamma(t(e))$ and $\bar{e}: \Gamma(e)
\rightarrow \Gamma(o(e))$ is an isomorphism. If $S$ contains a trivial
edge $e_1$ 
--- to fix ideas, say $\Gamma(e_1)^{e_1} = \Gamma(t(e_1))$ ---
then we contract the edge $e_1$ into the vertex $o(e_1)$ and re-define
incidence and embeddings where necessary, to obtain a new graph of
groups $(\Gamma'(-), X')$ with spanning tree $S'$ in $X'$. More
precisely, this means that we let  
\begin{align*}
E(X') &= E(X) \setminus \{e_1, \bar{e}_1\},\\[1mm]
E(S') &= E(S) \setminus \{e_1, \bar{e}_1\},\\[1mm]
V(X') &= V(S') = V(X)\setminus \{t(e_1)\},
\end{align*}
set 
\[
t'(e):= o(e_1), \quad \text {for }e\in E(X')\text { with }t(e) = t(e_1),
\] 
and define new embeddings via
\begin{equation}
\label{Eq:ImbedExt}
\Gamma(e) \overset{e}{\longrightarrow} \Gamma(t(e_1))
\overset{e_1^{-1}}{\longrightarrow} \Gamma(e_1)
\overset{\bar{e}_1}{\longrightarrow} \Gamma(o(e_1)) = \Gamma(t'(e)),
\quad \text {for }e\in E(X')\text { with } t(e)=t(e_1), 
\end{equation}
leaving incidence and embeddings unchanged wherever possible. Clearly,
$S'$, the result of contracting the geometric edge $\{e_1,
\bar{e}_1\}$ and deleting the vertex $t(e_1)$, is still a spanning
tree  for $X'$ and, if $(\Gamma(-),X)$ has property $(F_1)$ or
$(F_2)$, then so does $(\Gamma'(-), X')$ by construction. 

It remains to see that the fundamental group of the new graph of
groups $(\Gamma'(-), X')$ is isomorphic to $\Gamma$. The fundamental
group  
\[
\pi_1(\Gamma(-),X,S) 
\]
of the graph of groups $(\Gamma(-),X)$ at the 
spanning tree $S$ is
generated by the groups $\Gamma(v)$ for $v\in V(X)$ plus extra
generators $\gamma_e$ for $e\in \mathcal{O}(X)-E(S)$, where
$\mathcal{O}(X)$ is any orientation of $X$, subject to the relations 
\begin{align}
a^e &= a^{\bar{e}},\quad \text {for }e\in \mathcal{O}(S)\text { and } a\in
\Gamma(e),\label{Eq:GammaPres1}\\[1mm] 
\gamma_e a^e \gamma_e^{-1} &= a^{\bar{e}},\quad \text {for }e\in
\mathcal{O}(X)-E(S)\text { and } a\in \Gamma(e),\label{Eq:GammaPres2} 
\end{align}
where $\mathcal{O}(S)$ is the orientation of the tree $S$ induced by
$\mathcal{O}(X)$, with a corresponding presentation for
$\pi_1(\Gamma'(-),X',S')$;  
see \S5.1 in \cite[Chap.~I]{Serre2}.  The
relations~\eqref{Eq:GammaPres1} corresponding to the geometric edge
$\{e_1, \bar{e}_1\}$ identify $\Gamma(t(e_1))$ isomorphically with a
subgroup of $\Gamma(o(e_1))$; we can thus delete the generators
$\gamma\in \Gamma(t(e_1))$ against those relations by Tietze
moves. This yields a presentation for $\pi_1(\Gamma(-), X,S)$ with the
same set of generators as $\pi_1(\Gamma'(-), X',S')$. Moreover, those
relations \eqref{Eq:GammaPres1}--\eqref{Eq:GammaPres2} coming from
edges~$e$ with $t(e)=t(e_1)$ have to be re-expressed in terms of
elements of $\Gamma(o(e_1))$, which leads exactly to the corresponding
relations of $\pi_1(\Gamma'(-), X',S')$ obtained by extending the
embedding $e: \Gamma(e) \rightarrow \Gamma(t(e_1))$ in the natural way
as given in \eqref{Eq:ImbedExt}. Hence, $\pi_1(\Gamma(-), X,S) \cong
\pi_1(\Gamma'(-), X', S')$. Since $V(X)$ is finite, the tree $S$ is
finite; thus, proceeding in the manner described, we obtain, after
finitely many steps, a graph of groups $(\Delta(-), Y)$ with
fundamental group $\Gamma$ and a spanning tree $T$ in $Y$ without
trivial edges, such that $(\Delta(-),Y)$ enjoys the finiteness
properties $(F_1), (F_2)$ whenever $(\Gamma(-),X)$ does.  
\end{proof}

\section{A graph-theoretic lemma}
\label{Sec:TreeOrient}

 The following auxiliary result, which is of an entirely 
graph-theoretic nature, will be used frequently in 
the rest of the paper.

\begin{lemma}
\label{Lem:TreeOrient}
Let $T$ be a tree, and let $v_0\in V(T)$ be any vertex. Then there
exists one, and only one, orientation $\mathcal{O}(T)$ of $T,$ such
that the assignment $e\mapsto t(e)$ defines a bijection $\psi_{v_0}:
\mathcal{O}(T) \rightarrow V(T)\setminus\{v_0\}$. This orientation is
obtained by orienting each geometric edge so as to point away from the
root $v_0;$ that is, travelling along an edge of $\mathcal{O}(X),$ the distance
from $v_0$ in the path metric always  increases.  
\end{lemma}

Lemma~\ref{Lem:TreeOrient} is easy to show, even in this
generality. Moreover, for our present purposes, the trees considered
will all be finite, in which case the assertion of
Lemma~\ref{Lem:TreeOrient} may be proved by a straightforward
induction on $\vert V(T)\vert$, which we sketch briefly: by our
condition on the map $\psi_{v_0}$, all (geometric) edges incident with
$v_0$ will have to be oriented away from the root $v_0$.  Delete $v_0$
together with edges incident to $v_0$. The result is a disjoint union
of finitely many subtrees, in which we choose the (previous)
neighbours of $v_0$ as new roots. An application of the induction
hypothesis to these rooted subtrees now finishes the proof.  

\medskip
 In what follows, the orientation of a tree $T$ with respect to a base
point $v_0$ described in Lemma~\ref{Lem:TreeOrient} will be denoted by
$\mathcal{O}_{v_0}(T)$. 

\section{An inequality for the number of edges of a graph of groups}
\label{Sec:Ineq}

 An important consequence of normalisation is the following.

\begin{lemma}
\label{Lem:EEst}
Let $(\Gamma(-), X, T)$ be a finite graph of finite groups with maximal tree $T\leq X$ and fundamental group $\pi_1(\Gamma(-), X) \cong \Gamma$. If $(\Gamma(-), X, T)$ satisfies the normalisation condition {\em (\ref{eq:normal}),} then the number of edges $\vert E(X)\vert$ of the graph $X$ is bounded above in terms of the free rank of\/ $\Gamma$ via
\begin{equation}
\label{Eq:EEst}
\vert E(X)\vert \leq 2 \mu(\Gamma).
\end{equation}
\end{lemma}
\begin{proof}
We distinguish two cases.
\medskip

 (a) $\vert V(X)\vert = 1$. Then $m_\Gamma = \vert\Gamma(v)\vert$, where $V(X) = \{v\}$, and the Euler characteristic of $\Gamma$ becomes
\begin{align*}
\chi(\Gamma) &= \frac{1}{\vert\Gamma(v)\vert}\,-\, \sum_{e\in\mathcal{O}(X)} \frac{1}{\vert \Gamma(e)\vert}\\[1mm]
&= m_\Gamma^{-1} \bigg(1 - \sum_{e\in\mathcal{O}(X)} \big(\Gamma(v) : \Gamma(e)^e\big)\bigg)\\[1mm]
&\leq - m_\Gamma^{-1} \big(\vert\mathcal{O}(X)\vert - 1\big),
\end{align*}
where $\mathcal{O}(X)$ is an arbitrary orientation of $X$. It follows that
\[
\mu(\Gamma) = 1 - m_\Gamma\chi(\Gamma) \geq \vert \mathcal{O}(X)\vert,
\]
whence our claim in this case.
\medskip

 (b) $\vert V(X)\vert \geq2$. Then $E(T) \neq \emptyset$, and we may choose some edge $e_1\in E(T)$. Consider the tree $T$ as rooted with root $v_1 = o(e_1)$ and associated orientation $\mathcal{O}_{v_1}(T)$ in the sense of Lemma~\ref{Lem:TreeOrient}. Extending $\mathcal{O}_{v_1}(T)$ to an orientation $\mathcal{O}(X)$ of $X$, we write
\begin{multline*}
\chi(\Gamma) = \Big(\frac{1}{\vert\Gamma(o(e_1))\vert} + \frac{1}{\vert\Gamma(t(e_1))\vert} - \frac{1}{\vert \Gamma(e_1)\vert}\Big)\\[1mm] + \sum_{e\in\mathcal{O}_{v_1}(T)\setminus\{e_1\}}\Big(\frac{1}{\vert\Gamma(t(e))\vert} 
- \frac{1}{\vert\Gamma(e)\vert}\Big) 
- \sum_{e\in\mathcal{O}(X)\setminus\mathcal{O}_{v_1}(T)} \frac{1}{\vert\Gamma(e)\vert}.
\end{multline*}
Since the edge $e_1$ is not trivial, we have $2\vert\Gamma(e_1)\vert \leq \vert\Gamma(o(e_1))\vert$ as well as $2\vert\Gamma(e_1)\vert \leq \vert \Gamma(t(e_1))\vert$, thus
\[
\frac{1}{\vert\Gamma(o(e_1))\vert} + \frac{1}{\vert\Gamma(t(e_1))\vert}\, \leq\, \frac{1}{\vert\Gamma(e_1)\vert}. 
\]
For the same reason, for $e\in \mathcal{O}_{v_1}(T)\setminus\{e_1\}$, 
we have
\[
\frac{1}{\vert\Gamma(t(e))\vert} - \frac{1}{\vert\Gamma(e)\vert} = \frac{1-(\Gamma(t(e)) : \Gamma(e)^e)}{\vert\Gamma(t(e))\vert} \,\leq\, - \frac{1}{\vert\Gamma(t(e))\vert}\, \leq\, -\frac{1}{m_\Gamma}.
\]
Putting together these observations, we find that
\[
\chi(\Gamma) \leq -m_\Gamma^{-1}\big(\vert\mathcal{O}_{v_1}(T)\vert - 1\big) - m_\Gamma^{-1} \vert\mathcal{O}(X)\setminus\mathcal{O}_{v_1}(T)\vert = -m_\Gamma^{-1}\big(\vert\mathcal{O}(X)\vert - 1\big),
\]
from which our claim follows as before.
\end{proof}

\section{Classifying virtually infinite-cyclic groups}
\label{Sec:Rank1}
 Virtually infinite-cyclic groups play a
certain role in topology as they are precisely the finitely generated
groups with two ends. Their structure is well-known;
cf.\ \cite[5.1]{Stallings} or \cite[Lemma~4.1]{Wall}. In
this section, we shall give a short proof of the
corresponding result (Proposition~\ref{Prop:InfCyc}) based on the 
tools developed in 
Sections~\ref{Sec:GoGNorm} and \ref{Sec:TreeOrient}. As a
consequence of this classification result, we find that the function $f_\lambda(\Gamma)$ is constant for $\mu(\Gamma) = 1$; 
cf.\ Corollary~\ref{Cor:mu=1}.

\begin{proposition}
\label{Prop:InfCyc}
A virtually infinite-cyclic group $\Gamma$ falls into one of the
following two classes: 
\begin{enumerate}
\item[(i)] $\Gamma$ has a finite normal subgroup with infinite-cyclic quotient.
\vspace{2mm}

\item[(ii)] $\Gamma$ is a free product $\Gamma = G_1
\underset{A}{\ast} G_2$ of two finite groups $G_1$ and $G_2$, with an
amalgamated subgroup $A$ of index $2$ in both factors.  
\end{enumerate}
\end{proposition}

\begin{proof}
Let $(\Gamma(-),X)$ be a finite graph of finite groups with
fundamental group $\Gamma$ and spanning tree $T$, chosen according to
Lemma~\ref{Lem:Normalise}. The reader should observe that the
assumption that $\Gamma$ is 
virtually infinite-cyclic in combination with \eqref{Eq:FreeEuler} implies that
$\chi(\Gamma)=0$.  

 If $\vert V(X)\vert = 1$, $V(X)=\{v\}$ say, then
the above observation together with Formula~\eqref{Eq:EulerCharDecomp} 
shows that $X$ has exactly one
geometric edge $\{e, \bar{e}\}$, and that the associated embeddings $e, \bar{e}:
\Gamma(e)\rightarrow \Gamma(v)$ are isomorphisms. Hence, $\Gamma(v)
\unlhd \Gamma$ and $\Gamma/\Gamma(v) \cong C_\infty$, which gives the
desired result in Case~(i). 

 If $\vert V(X)\vert >1$, we choose an edge $e_1\in E(T)$, introduce
the orientation $\mathcal{O}_{v_0}(T)$ with respect to the base point
$v_0=o(e_1)$, extend it to an orientation $\mathcal{O}(X)$ of $X$,
and let $v_1=t(e_1)$. We then split the Euler characteristic of $\Gamma$ 
as follows: 
\begin{multline}
\label{Eq:chiSplit}
0 = \chi(\Gamma)  
= \underset{v\neq v_0, v_1}{\sum_{v\in V(X)}} \frac{1}{\vert
\Gamma(v)\vert}\,-\,\underset{e\neq
e_1}{\sum_{e\in\mathcal{O}_{v_0}(T)}} \frac{1}{\vert
\Gamma(e)\vert}\,\,+\,\,\Big(\frac{1}{\vert \Gamma(v_0)\vert} +
\frac{1}{\vert \Gamma(v_1)\vert} - \frac{1}{\vert
\Gamma(e_1)\vert}\Big) \\[1mm] 
-\sum_{e\in\mathcal{O}(X)\setminus \mathcal{O}_{v_0}(T)}
\frac{1}{\vert\Gamma(e)\vert}. 
\end{multline}
By the normalisation condition 
\eqref{eq:normal} 
on $(\Gamma(-), X, T)$, we have 
\[
2 \vert \Gamma(e_1)\vert\, \leq \gamma:=
\min\big\{\vert\Gamma(v_0)\vert, \,\Gamma(v_1)\vert\big\}, 
\]
so
\begin{equation}
\label{Eq:chi=0First}
\frac{1}{\vert \Gamma(v_0)\vert}\, +\,
\frac{1}{\vert\Gamma(v_1)\vert}\, -\, \frac{1}{\vert
\Gamma(e_1)\vert}\, \leq\, \frac{2}{\gamma}\, -\, \frac{1}{\vert
\Gamma(e_1)\vert}\, \leq\, 0. 
\end{equation}
Clearly, equality in \eqref{Eq:chi=0First} occurs if, and only if,
$\Gamma(e_1)$ is of index $2$ in both $\Gamma(v_0)$ and
$\Gamma(v_1)$. 
Similarly, by the normalisation condition \eqref{eq:normal}
and Lemma~\ref{Lem:TreeOrient}, we have
\[
\underset{v\neq v_0, v_1}{\sum_{v\in V(X)}} \frac{1}{\vert
\Gamma(v)\vert}\,-\,\underset{e\neq
e_1}{\sum_{e\in\mathcal{O}_{v_0}(T)}} \frac{1}{\vert
\Gamma(e)\vert}\,=\, \underset{e\neq
e_1}{\sum_{e\in\mathcal{O}_{v_0}(T)}}
\Big(\frac{1}{\vert\Gamma(t(e))\vert} -
\frac{1}{\vert\Gamma(e)\vert}\Big) \leq 0, 
\]
with equality if, and only if, $\mathcal{O}_{v_0}(T) = \{e_1\}$. Also,
trivially, the last sum on the right-hand side of \eqref{Eq:chiSplit}
is non-negative, and vanishes if, and only if,  
 $\mathcal{O}(X)=\mathcal{O}_{v_0}(T)$. Given this discussion, we
conclude from \eqref{Eq:chiSplit} that $\Gamma = \Gamma(v_0)
\underset{\Gamma(e_1)}{\ast} \Gamma(v_1)$, the amalgam being formed
with respect to the embeddings $e_1: \Gamma(e_1) \rightarrow
\Gamma(v_1)$ and $\bar{e}_1: \Gamma(e_1)\rightarrow \Gamma(v_0)$, and
that $(\Gamma(v_0):\Gamma(e_1)^{\bar{e}_1}) = 2 =
(\Gamma(v_1):\Gamma(e_1)^{e_1})$, whence the result in Case~(ii).  
\end{proof}

\begin{remarks}
{\em 1. In Case~(i) of Proposition~\ref{Prop:InfCyc}, we have $\zeta_\kappa=0$ for all $\kappa\mid m_\Gamma$ 
whereas, in Case~(ii), $\zeta_{m_\Gamma}=-1$. Hence, groups occurring in Case~(i) are not isomorphic to groups belonging to Case~(ii).}\label{Rem:InfCyc2}

\vspace{2mm}

{\em  2. In Part~(ii) of Proposition~\ref{Prop:InfCyc}, $A$ is a
finite normal subgroup of $\Gamma$ with quotient $C_2\ast C_2$, the infinite dihedral group.} \label{Rem:InfCyc1}
\end{remarks}

\begin{corollary}
\label{Cor:mu=1}
If\/ $\Gamma$ is  virtually infinite-cyclic, then the function
$f_\lambda(\Gamma)$ is constant. More precisely, we have
$f_\lambda(\Gamma)=m_\Gamma$ for $\lambda\geq1$ in   Case~{\em (i)} of
Proposition~{\em \ref{Prop:InfCyc},} while in Case~{\em (ii)} we have
$f_\lambda(\Gamma) = \vert A\vert = m_\Gamma/2$. 
\end{corollary}
\begin{proof}
If $\Gamma$ is as described in Case~(i) of
Proposition~\ref{Prop:InfCyc}, then \eqref{Eq:gExplicit} shows that
$g_\lambda(\Gamma)=1$ for $\lambda\geq0$, leading to
$f_\lambda(\Gamma)=m_\Gamma$ for all $\lambda\geq	1$ by
\eqref{Eq:Transform} and an immediate induction on $\lambda$. 

For $\Gamma$ as in Case~(ii), Equation~\eqref{Eq:gExplicit} yields 
\[
g_\lambda(\Gamma) = 2^{-2\lambda} \binom{2\lambda}{\lambda},\quad \lambda\geq0.
\]
By the binomial theorem applied to the generating function 
$G_\Gamma(z)$ of the $g_\lambda(\Gamma)$'s, we obtain 
$G_\Gamma(z)=(1-z)^{-1/2}$,
which transforms into
the relation
\[
F_\Gamma(z) 
= \frac{\vert m_\Gamma\vert}{2(1-z)}
= \frac{\vert A\vert}{1-z}
\]
via \eqref{Eq:GFTransform}. The desired result follows from this last
equation by  comparing coefficients. 
\end{proof}

\section{The case where $\mu(\Gamma)=2$}
\label{Sec:mu=2}
\subsection{The classification result}
\begin{proposition}
\label{Prop:mu=2}
A virtually free group $\Gamma$ of rank $\mu(\Gamma)=2$ falls into one
of the following five classes: 
\begin{enumerate}
\item[(i)] $\Gamma$ is an HNN-extension $\Gamma =
G\underset{A,\phi}{\ast}$ with finite base group $G,$ associated
subgroups $A$ and $B=\phi(A),$ associated isomorphism $\phi:
A\rightarrow B,$ and $(G:A)=2$.  
\vspace{2mm}

\item[(ii)] $\Gamma$ contains a finite normal subgroup $G$ with
quotient $\Gamma/G \cong F_2$ free of rank $2$. 
\vspace{2mm}

\item[(iii)] $\Gamma$ is a free product\/ $\Gamma = G_1
\underset{S}{\ast} G_2$ of two finite groups $G_i$ with an amalgamated
subgroup $S,$ whose indices $(G_i:S)$ satisfy one of the conditions 

\vspace{2mm}

\begin{enumerate}
\item[(iii)$_1$] $\{(G_1:S),\,(G_2:S)\} = \{2,3\},$
\vspace{1mm}

\item[(iii)$_2$] $(G_1:S) = 3 = (G_2:S),$
\vspace{1mm}

\item[(iii)$_3$] $\{(G_1:S),\, (G_2:S)\} = \{2, 4\}$.
\end{enumerate}

\vspace{2mm}

\item[(iv)] $\Gamma$ is a free product $\Gamma = G_1\underset{S}{\ast}
  \Gamma_2,$ where $G_1$ is finite, $\Gamma_2$ is a virtually
  infinite-cyclic group of type {\em (i)} {\em(}see Proposition~{\em
    \ref{Prop:InfCyc}}{\em)}, and $(G_1:S) = 2 = (G_2:S),$ where $G_2$ is
  the base group of the HNN-extension $\Gamma_2$. 

\vspace{2mm}
 
\item[(v)] $\Gamma$ is of the form $\Gamma = (G_1
\underset{S_1}{\ast} G_2) \underset{S_2}{\ast} G_3$ with finite
factors $G_1, G_2, G_3$ and subgroups $S_1, S_2$ 
satisfying $\vert G_1\vert = \vert G_2\vert = \vert G_3\vert = 2\vert
S_1\vert = 2\vert S_2\vert$. 
\end{enumerate}
\end{proposition} 

\begin{proof}
Let $\Gamma$ be a virtually free group of free rank $\mu(\Gamma)=2$,
let $(\Gamma(-), X)$ be a Stallings decomposition of $\Gamma$, and let
$T$ be a spanning tree in $X$ satisfying the normalisation condition
(\ref{eq:normal}) of Lemma~\ref{Lem:Normalise}. By
Lemma~\ref{Lem:EEst}, $X$ has at most
two geometric edges, while, by Equation (\ref{Eq:FreeEuler}), we have
$\chi(\Gamma) = - \frac{1}{m_\Gamma}$. There are five possibilities
for the isomorphism type of the graph $X$ underlying the decomposition
of $\Gamma$,  
and the proof of the proposition (as well as its statement) breaks
into cases accordingly. 

\vspace{2mm}

 (i) \textit{$X$ consists of a single loop $e$ with
  $o(e)=t(e)=v$.} Setting $G:= \Gamma(v)$ and $S:= \Gamma(e)$, we have
$m_\Gamma = \vert G\vert$ and 
\[
\chi(\Gamma) = \frac{1}{\vert G\vert} - \frac{1}{\vert S\vert} =
\frac{1 - (G:S^e)}{m_\Gamma} = -\frac{1}{m_\Gamma}, 
\]
implying $(G:S^e)=2$. Thus, setting $A:=S^e$, $B:=S^{\bar{e}}$, and
with the isomorphism $\phi: A \rightarrow B$ given by $x^e \mapsto
x^{\bar{e}}$ (in keeping with the notation of \cite[Chap.~IV.2]{LynSchu}),
the definition of $\pi_1(\Gamma(-), X)$ yields that  
\[
\Gamma \cong \big\langle G, t\,\big\vert \,tat^{-1} = \phi(a),\, a\in
A\big\rangle ,
\]
whence the result in that case.

\vspace{2mm}

 (ii) \textit{$X$ consists of a single vertex $v,$ supporting
  two loops $e_i$, $i=1,2$.} Set $G:= \Gamma(v)$ and $S_i:=
\Gamma(e_i)$. Then $m_\Gamma = \vert G\vert$, and 
\[
\chi(\Gamma) = \frac{1}{\vert G\vert} - \frac{1}{\vert S_1\vert} -
\frac{1}{\vert S_2\vert} = \frac{1 - (G:S_1^{e_1}) -
  (G:S_2^{e_2})}{m_\Gamma} = -\frac{1}{m_\Gamma}, 
\]
implying
\[
(G: S_1^{e_1}) = 1 = (G: S_2^{e_2}).
\]
Hence, the maps $e_i: S_i \rightarrow G$ are isomorphisms, and we
obtain the presentation  
\[
\Gamma \cong \big\langle G, s_1, s_2\,\big\vert\,s_1 a_1^{e_1}
s_1^{-1} = a_1^{\bar{e}_1}\, (a_1\in S_1),\, s_2 a_2^{e_2} s_2^{-1} =
a_2^{\bar{e}_2}\, (a_2\in S_2)\big\rangle. 
\]
It follows that the finite group $G$ is normal in $\Gamma$ with
quotient a free group of rank two, as claimed. 

\vspace{2mm}

 (iii) \textit{$X=T$ is a segment $e$ with vertices $v_1,
  v_2,$ say $t(e)=v_2$.} Set $G_i:= \Gamma(v_i), \,i=1,2,$ and $S:=
\Gamma(e)$. Then $\Gamma = G_1 \underset{S}{\ast} G_2$, with the
canonical embeddings given by $\bar{e}: S\rightarrow G_1$ and $e:
S\rightarrow G_2$.  Moreover, let $a_1:= (G_1:S^{\bar{e}})$ and $a_2:=
(G_2:S^e)$. By symmetry, we may suppose that $a_1\leq a_2$, we have
$a_1\geq2$ by our assumption that $(\Gamma(-), X, T)$ is normalised,
and the requirement that $\mu(\Gamma)=2$ boils down to the
(equivalent) equation 
\begin{equation}
\label{Eq:RankCase3}
a_1a_2 - a_1 - a_2 = \gcd(a_1, a_2).
\end{equation}
Since $\gcd(a_1, a_2) \leq a_1$, Equation~(\ref{Eq:RankCase3}) implies that
\begin{equation}
\label{Eq:RankCase3Ineq}
a_1 \leq a_2 \leq \frac{2a_1}{a_1-1},
\end{equation}
which in turn leads to $a_1^2\leq 3a_1$. Given our present
constraints, the last inequality is satisfied only for $a_1 = 2$ and
$a_1 = 3$. If $a_1=2$, then we find from (\ref{Eq:RankCase3Ineq}) that
$2\leq a_2\leq 4$, while, for $a_1=3$, we get $a_2=3$. Thus, the only
possibilities are 
\[
(a_1, a_2) = (2,2),\, (2,3),\, (2,4),\, (3,3),
\]
and, inserting these into (\ref{Eq:RankCase3}), the possible solution
$a_1=2=a_2$ is eliminated, while the remaining three pairs all solve
(\ref{Eq:RankCase3}), whence the result in that case. 

\vspace{2mm}

 (iv) \textit{$X$ consists of a segment $e_1$ with vertices
  $v_1$ and $v_2,$ say $t(e_1) = v_2,$ with a loop $e_2$ attached at
  $v_2$.} For $i=1,2$, set $G_i:= \Gamma(v_i)$, and let $S_i:=
\Gamma(e_i)$. Then $\Gamma = G_1 \underset{S_1}{\ast} \Gamma_2$, where
$\Gamma_2$ is the fundamental group of the loop $e_2$ with bounding
vertex $v_2$, and the canonical embeddings are given by the maps
$\bar{e}_1: S_1\rightarrow G_1$ and $\tilde{e}_1:
S_1\overset{e_1}{\rightarrow} G_2\rightarrow \Gamma_2$. Let $a_1:=
(G_1:S_1^{\bar{e}_1})$, $a_2:= (G_2:S_1^{e_1})$, and $a_2':=
(G_2:S_2^{e_2})$. Then 
\[
m_\Gamma = \lcm\!\big(\vert G_1\vert, \vert G_2\vert\big) 
= \vert S_1\vert\cdot \lcm(a_1, a_2) 
= \vert S_2\vert \cdot \lcm(a_1, a_2) a_2'/a_2,
\] 
and the condition that $\mu(\Gamma)=2$ translates into the equation
\begin{equation}
\label{Eq:RankCase4}
a_1a_2 + a_1a_2' - a_1 - a_2 = \gcd(a_1, a_2).
\end{equation}
Moreover, we have $a_1, a_2\geq2$ by our assumption that $(\Gamma(-),
X, T)$ is normalised, where $T$ is the unique spanning tree of
$X$. Suppose first that $a_1 \leq a_2$. Then (\ref{Eq:RankCase4})
gives 
\[
2 \leq a_1 \,\leq\, a_2 \leq \frac{(2-a_2')a_1}{a_1-1} \,\leq
\,\frac{a_1}{a_1-1}\, \leq\, 2. 
\]
This forces $a_1=a_2=2$, and from (\ref{Eq:RankCase4}) we deduce that
$a_2'=1$. Now suppose that $a_1\geq a_2$. Then (\ref{Eq:RankCase4})
yields 
\[
2 \leq a_2 \,\leq\, a_1 \,\leq \, \frac{2a_2}{a_2 + a_2'-1} \, \leq \, 2,
\]
which again leads to the solution $a_1 = a_2 = 2$ and
$a_2'=1$. Assertion (iv) now follows. 

\vspace{2mm}

(v) \textit{$X=T$ is a path $(v_1, e_1, v_2, e_2, v_3)$ of length
  $2$.} For $i = 1,2,3$, set $G_i:=\Gamma(v_i)$, and let
$S_j:=\Gamma(e_j)$ for $j=1,2$. Then $\Gamma = \big(G_1
\underset{S_1}{\ast} G_2\big) \underset{S_2}{\ast} G_3$. Since
$\mu(\Gamma)=2$, we have 
\begin{equation}
\label{Eq:RankCase5}
\frac{m_\Gamma}{\vert S_1\vert}\, + \,\frac{m_\Gamma}{\vert
  S_2\vert}\, -\, \frac{m_\Gamma}{\vert G_1\vert}\, -
\,\frac{m_\Gamma}{\vert G_2\vert}\, -\, \frac{m_\Gamma}{\vert
  G_3\vert} = 1. 
\end{equation}
As $(\Gamma(-), X, T)$ is normalised, we have
\[
\frac{m_\Gamma}{\vert G_1\vert} \,\leq\, \frac{m_\Gamma}{2\vert
  S_1\vert} \,\mbox{ and }\, \frac{m_\Gamma}{\vert G_2\vert}\, \leq\,
\frac{m_\Gamma}{2\vert S_2\vert}, 
\]
so that (\ref{Eq:RankCase5}) gives
\[
\frac{m_\Gamma}{2\vert S_1\vert} \,+\,\frac{m_\Gamma}{2\vert S_2\vert}
- \frac{m_\Gamma}{\vert G_3\vert}\, \leq \,1. 
\]
Again by normalisation,
\[
\frac{m_\Gamma}{\vert G_3\vert} \,\leq \, \frac{m_\Gamma}{2\vert S_2\vert},
\]
thus
\[
1\,\leq\,\max\left\{\frac{\vert G_1\vert}{2\vert S_1\vert},\,
\frac{\vert G_2\vert}{2\vert S_1\vert}\right\} \,\leq\,
\frac{m_\Gamma}{2\vert S_1\vert} \,\leq \, 1, 
\]
implying
\[
m_\Gamma = \vert G_1\vert = \vert G_2\vert = 2\vert S_1\vert. 
\]
Using this information in (\ref{Eq:RankCase5}), we now find that
\[
m_\Gamma\left(\frac{1}{\vert S_2\vert} \,-\,\frac{1}{\vert
  G_3\vert}\right) = 1, 
\]
implying first $m_\Gamma = 2 \vert S_2\vert$ by normalisation, and
then $\vert G_3\vert = m_\Gamma$. 
\end{proof}

\begin{remark}
{\em By considering the type and the number of conjugacy classes of maximal finite subgroups, one shows again that any two groups from different classes in Proposition~\ref{Prop:mu=2} are not isomorphic.}
\end{remark}

\subsection{Some consequences of Proposition~\ref{Prop:mu=2}}

Using the structural classification afforded by
Proposition~\ref{Prop:mu=2} in conjunction with
(\ref{Eq:GFTransform}), (\ref{Eq:G(z)Diff}), and
(\ref{Eq:G(z)DiffCoeffs}), we obtain, for each of the five cases in
Proposition~\ref{Prop:mu=2}, a recurrence relation for the
corresponding function $f_\lambda(\Gamma)$. The result is as follows: 

\vspace{2mm}

 (a) In Cases~(i) and (iv),
\begin{equation}
\label{Eq:mu=2Rec1}
f_{\lambda+1}(\Gamma) = \frac{2\lambda+3}{2} m_\Gamma
f_\lambda(\Gamma) \,+\,\sum_{\mu=1}^{\lambda-1} f_\mu(\Gamma)
f_{\lambda-\mu}(\Gamma),\quad \lambda\geq1.
\end{equation}

\vspace{2mm}

 (b) In Case~(ii),
\begin{equation}
\label{Eq:mu=2Rec2}
f_{\lambda+1}(\Gamma) = (\lambda+2) m_\Gamma
f_\lambda(\Gamma)\,+\,\sum_{\mu=1}^{\lambda-1} f_\mu(\Gamma)
f_{\lambda-\mu}(\Gamma),\quad \lambda\geq1.
\end{equation}

\vspace{2mm}

 (c) In Cases~(iii) and (v),
\begin{equation}
\label{Eq:mu=2Rec3}
f_{\lambda+1}(\Gamma) = (\lambda+1) m_\Gamma
f_\lambda(\Gamma)\,+\,\sum_{\mu=1}^{\lambda-1} f_\mu(\Gamma)
f_{\lambda-\mu}(\Gamma),\quad \lambda\geq1, 
\end{equation}
with corresponding initial conditions
\begin{enumerate}
\item[(a)] $f_1(\Gamma) = m_\Gamma^2/2$,

\vspace{2mm}

\item[(b)] $f_1(\Gamma) = m_\Gamma^2$,

\vspace{2mm}

\item[(c)] $f_1(\Gamma) = \begin{cases}
(m_\Gamma-\vert S\vert)\vert S\vert,& \mbox{Case~(iii)},\\
(m_\Gamma/2)^2,& \mbox{Case~(v).}
\end{cases}$
\end{enumerate}

We record two applications of
Equations~(\ref{Eq:mu=2Rec1})--(\ref{Eq:mu=2Rec3}) (and their initial
conditions). 

\begin{corollary}
\label{Cor:mu=2GrowthEst}
For a virtually free group $\Gamma$ with $\mu(\Gamma)=2$ and
$\Gamma\not\cong C_2\ast C_2\ast C_2,$ we have 
\begin{equation}
\label{Eq:mu=2GrowthEst}
f_{\lambda+1}(\Gamma) - f_\lambda(\Gamma)\,\geq\, m_\Gamma (\lambda + 1)!
\end{equation}
for all $\lambda\geq1$. For $\Gamma \cong C_2\ast C_2\ast C_2,$ the
estimate {\em (\ref{Eq:mu=2GrowthEst})} holds for all $\lambda\geq2$. 
\end{corollary}
\begin{proof}
This follows from the above recurrence relations plus initial
conditions by an immediate induction on $\lambda$. 
\end{proof}

\begin{corollary}
\label{Cor:mu=2Parity}
Let $\Gamma$ be virtually free of rank $\mu(\Gamma)=2$. In the cases
{\em (iii)$_1,$} {\em (iii)$_3,$} and {\em (v)}, we have, with $\vert
S\vert \equiv 1\ (\text{\em mod }2)$ respectively $\vert S_1\vert
\equiv 1\ (\text{\em mod }2),$  
\[
f_\lambda(\Gamma) \equiv 1\ (\text{\em mod }2)\quad \mbox{  if, and
  only if,} \quad \lambda=2^m-1 \mbox{ for some integer } m\geq1. 
\]
In all other cases, the function $f_\lambda(\Gamma)$ is constant modulo $2$. 
\end{corollary}
\begin{proof}
We focus on Case~(iii)$_1$ with $\vert S\vert\equiv
1\ (\text{mod }2)$; the proof in Cases~(iii)$_3$ and (v) is completely
analogous, while the fact that $f_\lambda(\Gamma)$ is constant modulo
$2$ in all other cases is immediate. 

We denote by $\Lambda$ the set of integers of the form
$\lambda = 2^m-1,$ $m=1, 2, \ldots$, and prove the equivalence 
\begin{equation}
\label{Eq:mu=2IntParity}
f_\lambda(\Gamma) \equiv 1\ (\text{mod }2)
\quad \text{if, and only if,}\quad \lambda\in\Lambda,
\end{equation}
for $\lambda\geq1$ by induction on $\lambda$. The assumption that
$\vert S\vert \equiv 1\ (\text{mod }2)$ implies that 
\[
f_1(\Gamma) = 5\vert S\vert^2 \equiv 1\ (\text{mod }2),
\]
so that (\ref{Eq:mu=2IntParity}) is true for $\lambda=1$. Suppose that
(\ref{Eq:mu=2IntParity}) is true for all $\lambda\leq L$ with some
$L\geq1$, and consider $\lambda=L+1$. From (\ref{Eq:mu=2Rec3}) and the
fact that $m_\Gamma =6\vert S\vert \equiv 0\ (\text{mod }2)$, we infer
that, for $\lambda\geq1$, 
\[
f_{\lambda+1}(\Gamma) \equiv \left.\begin{cases}
f_{\lambda/2}(\Gamma)\ (\text{mod }2),& 2\mid \lambda,\\[1mm]
0\ (\text{mod }2),& 2\nmid \lambda.\end{cases}\right.
\]
If $L+1\in \Lambda$, i.e., $L+1 = 2^m-1$ for some $m\geq2$, then
$L=2(2^{m-1}-1)\equiv 0\ (\text{mod }2)$ and $L/2\in\Lambda$, thus
$f_{L+1}(\Gamma) \equiv 1\ (\text{mod }2)$ by the induction
hypothesis. Suppose, on the other hand, that $L+1\not\in \Lambda$. If
$L\equiv 1\ \text{mod }2)$, then $f_{L+1}(\Gamma) \equiv
0\ \text{mod }2)$. Thus we are left with the case where
$L+1\not\in\Lambda$ and $L\equiv 0\ \text{mod }2)$. But then
$L/2\not\in \Lambda$, and the induction hypothesis gives  
\[
f_{L+1}(\Gamma) \equiv f_{L/2}(\Gamma) \equiv 0\ (\text{mod }2),
\]
completing the proof.
\end{proof}

Corollary~\ref{Cor:mu=2Parity} serves well to illustrate
the main result of \cite{KraMu}:

\smallskip
 (1) In Case~(i) of
Proposition~\ref{Prop:mu=2}, we have $2\mid m_\Gamma$ and  
\[
\mu_2(\Gamma) = \begin{cases}
1,&\vert A\vert \equiv 1\ (\text{mod }2),\\
2,& \vert A\vert \equiv 0\ (\text{mod }2).\end{cases}
\]
In particular, we obtain that $\mu_2(\Gamma)>0$,
so Case~(III)$_2$ of \cite[Theorem~1]{KraMu} applies, asserting that
$f_\lambda(\Gamma)$ is ultimately periodic modulo $2$ in this case.
Indeed, by
Corollary~\ref{Cor:mu=2Parity}, $f_\lambda(\Gamma)$ is constant modulo
$2$. 

\smallskip
 (2) In Case~(ii) of Proposition~\ref{Prop:mu=2}, we either have
$2\nmid m_\Gamma$, or $2\mid m_\Gamma$ and $\mu_2(\Gamma)=2>0$, so
$f_\lambda(\Gamma)$ is ultimately periodic modulo $2$ in this case
according to Case~(III)$_1$ respectively (III)$_2$ of
\cite[Theorem~1]{KraMu}. Indeed, $f_\lambda(\Gamma)$ is again constant modulo
$2$ by Corollary~\ref{Cor:mu=2Parity}. 

\smallskip
 (3) In Case~(iii)$_1$ of
Proposition~\ref{Prop:mu=2}, we have $2\mid m_\Gamma$ and 
$$\mu_2(\Gamma) = \begin{cases}
0,&\vert S\vert \equiv 1\ (\text{mod }2),\\
2, & \vert S\vert \equiv 0\ (\text{mod }2),\end{cases}$$ 
so
$f_\lambda(\Gamma)$ is ultimately periodic modulo $2$ according to
\cite[Theorem~1]{KraMu} if, and only if, $\vert S\vert \equiv
0\ (\text{mod }2)$, which coincides with the corresponding assertion
of Corollary~\ref{Cor:mu=2Parity}. 

\smallskip
 (4) In Case~(iii)$_2$ of
Proposition~\ref{Prop:mu=2}, we either have $\vert S\vert \equiv
1\ (\text{mod }2)$, and so $2\nmid m_\Gamma = 3\vert S\vert$, or
$\vert S\vert \equiv 0\ (\text{mod }2)$, in which case $2\mid
m_\Gamma$ and $\mu_2(\Gamma) =2>0$. Hence, ultimate periodicity of the
function $f_\lambda(\Gamma)$ modulo $2$ follows again from
Case~(III)$_1$ 
 respectively Case~(III)$_2$
 of \cite[Theorem~1]{KraMu}, while Corollary~\ref{Cor:mu=2Parity}
 asserts that $f_\lambda(\Gamma)$ is constant modulo $2$ in that
 case. 

\smallskip
 (5) In Case~(iii)$_3$ of Proposition~\ref{Prop:mu=2}, we have
 $2\mid m_\Gamma = 4\vert S\vert$ and 
\[
\mu_2(\Gamma) = \begin{cases}
0,& \vert S\vert \equiv 1\ (\text{mod }2),\\
2,& \vert S\vert \equiv 0\ (\text{mod }2).\end{cases}
\]
Hence, according to \cite[Theorem~1]{KraMu}, the function
$f_\lambda(\Gamma)$ is ultimately periodic modulo $2$ if, and only if,
$\vert S\vert \equiv 0\ (\text{mod }2)$, which is in accordance with the corresponding assertion of Corollary~\ref{Cor:mu=2Parity}. 

\smallskip
 (6) In Case~(iv) of Proposition~\ref{Prop:mu=2}, we have $2\mid m_\Gamma = 2\vert S\vert$ and 
\[
\mu_2(\Gamma) = \begin{cases}
1,& \vert S\vert \equiv 1\ (\text{mod }2),\\
2, & \vert S\vert \equiv 0\ (\text{mod }2).\end{cases}
\]
In particular, we obtain that $\mu_2(\Gamma)>0$,
so that ultimate periodicity of $f_\lambda(\Gamma)$ follows from
Case~(III)$_2$ of \cite[Theorem~1]{KraMu}, while
Corollary~\ref{Cor:mu=2Parity} asserts that $f_\lambda(\Gamma)$ is
constant modulo~$2$ in that case. 

\smallskip
 (7) Finally, in Case~(v) of
Proposition~\ref{Prop:mu=2}, we have $2\mid m_\Gamma = 2\vert
S_1\vert$ and 
\[
\mu_2(\Gamma) = \begin{cases}
0, & \vert S_1\vert \equiv 1\ (\text{mod }2),\\
2,& \vert S_1\vert \equiv 0\ (\text{mod }2).\end{cases}
\]
Hence, according to \cite[Theorem~1]{KraMu}, the function
$f_\lambda(\Gamma)$ is ultimately periodic modulo $2$ if, and only if,
$\vert S_1\vert \equiv 0\ (\text{mod }2)$, in accordance with the
corresponding assertion of Corollary~\ref{Cor:mu=2Parity}. 

\section{Some criteria for a virtually free group to be `large'}
\label{Sec:Large}

 Our final result collects together a number of equivalent
conditions on 
a finitely generated virtually free group $\Gamma$ which all say, in
one way or another, that $\Gamma$ is `large' in some particular
sense. Perhaps the most obvious condition in this direction is given
by Pride's concept of being `as large as a free group of rank
$2$'. The concept of `largeness' for groups, first introduced by
S. Pride in \cite{Pride}, and further developed in \cite{EP}, depends
on a certain preorder $\preceq$ on the class of groups, defined in
\cite{EP} as follows: let $G$ and $H$ be groups. Then we write
$H\preceq G$, if there exist 
\begin{enumerate}
\item[(a)] a subgroup $G^0$ of finite index in $G$;
\vspace{2mm}

\item[(b)] a subgroup $H^0$ of finite index in $H$, and a finite
normal subgroup $N^0$ of $H^0$; 
\vspace{2mm}

\item[(c)] a homomorphism from $G^0$ onto $H^0/N^0$.
\end{enumerate}
We write $H\sim G$ if $H\preceq G$ and $G\preceq H$, and we denote by
$[G]$ the equivalence class of the group $G$ under $\sim$. By abuse of
notation, we also denote by $\preceq$ the preorder induced on the
class of equivalence classes of groups. The finitely generated groups
which are `largest' in Pride's sense are the ones having a subgroup of
finite index which can be mapped homomorphically onto the free group
$F_2$ of rank $2$. 

 Another, more topological, way of saying that a finitely generated
virtually free group is `large', is that it has infinitely many
ends. Here, the number $e(\Gamma)$ of ends of a group $\Gamma$ is
defined as 
\[
e(\Gamma) = \begin{cases}
\dim H^0(\Gamma, \mathrm{Hom}_\Z(\Z\Gamma, \Z_2)/\Z_2\Gamma),&
\mbox{if $\Gamma$ is infinite,}\\[2mm] 
0,& \mbox{if $\Gamma$ is finite.}
\end{cases}
\]
The reader is referred to \cite{Cohen1} or \cite[Sec.~2]{Cohen2} for
an  introduction to the theory of ends of a group from an algebraic
point of view; for a discussion from a more topological viewpoint,
see, for instance, \cite{Freud}, \cite{Hopf}, or \cite{Specker}.  

\begin{proposition}
\label{Prop:chi<0Equivs}
Let\/ $\Gamma$ be a finitely generated virtually free group, and let
$(\Gamma(-), X)$ be a finite graph of finite groups with fundamental
group $\Gamma,$ chosen so as to satisfy the normalisation condition 
\eqref{eq:normal} of
Lemma~{\em \ref{Lem:Normalise}}. Then the following assertions on
$\Gamma$ are equivalent:

\begin{enumerate}
\item[(i)] $\chi(\Gamma) < 0$.
\vspace{2mm}

\item[(ii)] $\mu(\Gamma) \geq2$.
\vspace{2mm}

\item[(iii)] $\Gamma$ has infinitely many ends.
\vspace{2mm}

\item[(iv)] The function $f_\lambda(\Gamma)$ is strictly increasing.
\vspace{2mm}

\item[(v)] $\Gamma\sim F_2$ in the sense of Pride's preorder $\preceq$
on groups, where $F_2$ denotes the free group of rank $2$. 
\vspace{2mm}

\item[(vi)] $\Gamma$ has fast subgroup growth in the sense that the
inequality $s_{nj}(\Gamma) \geq c\cdot n!$ holds for some fixed positive
integer $j,$  some constant $c>0,$ and all $n\geq1$. Here $s_m(\Gamma)$ denotes the number of
subgroups of index $m$ in $\Gamma$. 
\vspace{2mm}

\item[(vii)] If $X$ has only one vertex $v,$ then either $X$ has more
than one geometric edge, or $E(X) = \{e_1, \bar{e}_1\}$ and
$(\Gamma(v): \Gamma(e_1)^{e_1}) \geq2;$ if $\vert V(X)\vert \geq2,$
then $X$ is not a tree, or $X$ is a tree with more than one geometric
edge, or $E(X) = \{e_1, \bar{e}_1\}$ and $\chi(\Gamma_0)<0,$ where
$\Gamma_0:= \Gamma_{o(e_1)} \underset{\Gamma(e_1)}{\ast}
\Gamma_{t(e_1)}$. 
\end{enumerate}
\end{proposition}

\begin{proof} 
(i) $\Leftrightarrow$ (ii). This is immediate from
Formula~\eqref{Eq:FreeEuler} plus the fact that $\mu(\Gamma)$ is
integral. 

\smallskip
 (ii) $\Leftrightarrow$ (iii). This follows from
\cite[Prop.~2.1]{Cohen2} (i.e., the fact that the number of ends is
invariant when passing to a subgroup of finite index) and Examples~1
and 2 in \cite{Cohen2} computing the number of ends of a free product,
respectively of $C_\infty$.  

\smallskip
 (ii) $\Leftrightarrow$ (iv). This follows from
\cite[Theorem~4]{MuDiss} in conjunction with
Corollary~\ref{Cor:mu=1}. 

\smallskip
 (ii) $\Rightarrow$ (v). If $\mu(\Gamma)\geq2$, then $\Gamma$ contains
a free group $F$ of rank at least $2$, with $(\Gamma:F) =
m_\Gamma<\infty$; in particular, $F_2 \preceq \Gamma$. 
Since $[F_2]$ is largest with
respect to the preorder $\preceq$ among all equivalence classes of
finitely generated groups, we also have $\Gamma \preceq F_2$, so
$\Gamma \sim F_2$, as claimed. 

\smallskip
 (v) $\Rightarrow$ (vi). Suppose that $\Gamma\sim F_2$. Then there
exists a subgroup $\Delta\leq \Gamma$ of index
$(\Gamma:\Delta)=j<\infty$ and a surjective homomorphism $\varphi:
\Delta\rightarrow F_2$. From this plus Newman's asymptotic estimate
\cite[Theorem~2]{Newman} 
\[
s_n(F_r) \sim n (n!)^{r-1}\mbox{ as }n\rightarrow\infty,\quad r\geq2,
\]
it follows that
\[
s_{jn}(\Gamma) \geq s_n(\Delta) \geq s_n(F_2) \geq c\cdot n\cdot n! \geq
c\cdot n!
\]
for $n\geq1$ and some constant $c>0$, whence (vi).

\smallskip
 (vi) $\Rightarrow$ (ii). If $\mu(\Gamma)\leq 1$, then either $\Gamma$
is finite, so $s_n(\Gamma)=0$ for sufficiently large $n$, or $\Gamma$
is virtually infinite-cyclic, implying 
\[
s_n(\Gamma) \leq n^\alpha,\quad n\geq1,
\]
for some constant $\alpha$, by \cite[Cor.~1.4.3]{LS}; see also
\cite{Segal}. In both cases, Condition~(vi) does not hold. 

\smallskip
 (ii) $\Leftrightarrow$ (vii). This follows by splitting the Euler
characteristic $\chi(\Gamma)$ as in the proof of
Proposition~\ref{Prop:InfCyc}, making use of
Lemmas~\ref{Lem:Normalise} and ~\ref{Lem:TreeOrient}. 
\end{proof}


\begin{thebibliography}{99}
\bibitem{Brown} K.\,S. Brown, \textit{Cohomology of groups},
Springer-Verlag, New York, 1982. 

\bibitem{Cohen1} D.\,E. Cohen, Ends and free products of groups,
\textit{Math. Zeitschr.} \textbf{114} (1970), 9--18. 

\bibitem{Cohen2} D.\,E. Cohen, \textit{Groups of Cohomological
Dimension One}, Lecture Notes in Mathematics, vol.~245,
Springer-Verlag, Berlin--Heidelberg--New York, 1972. 
 
\bibitem{DM} A. Dress and T.\,W. M\"uller, Decomposable functors and
the exponential principle, \textit{Adv. Math.} \textbf{129} (1997),
188--221. 

\bibitem{EP} M. Edjvet and S.\,J. Pride, The concept of ``largeness''
in group theory II. In: \textit{Proc. Groups-Korea 1983}, Lecture
Notes in Math., vol.~1098, Springer-Verlag, Berlin-Heidelberg-New York,
1985, pp.~29--54. 

\bibitem{Freud} H. Freudenthal, \"Uber die Enden diskreter R\"aume und
Gruppen, \textit{Comment. Math. Helv.} \textbf{17} (1944), 1--38. 

\bibitem{Hopf} H. Hopf, Enden offener R\"aume und unendliche
diskontinuierliche Gruppen, \textit{Comment. Math. Helv.} \textbf{16}
(1943), 81--100. 





\bibitem{KraMu} C. Krattenthaler and T.\,W. M\"uller, Periodicity of
  free subgroup numbers modulo prime powers, \textit{J. Algebra}
  \textbf{452} (2016), 372--389. 

\bibitem{KrMuAJ}
C. Krattenthaler and T.\,W. M\"uller, Free 
subgroup numbers modulo prime powers: the non-periodic case, preprint;
{\tt ar$\chi$iv.org:1602.08723}.

\bibitem{LS} A. Lubotzky and D. Segal, \textit{Subgroup Growth},
Progress in Mathematics, vol.~212, Birkh\"auser--Verlag,
Basel-Boston-Berlin, 2003. 

\bibitem{LynSchu} R.\,C. Lyndon and P.\,E. Schupp,
  \textit{Combinatorial Group Theory}, Springer-Verlag
  Berlin-Heidelberg-New York, 1977. 

\bibitem{MuEJC} T.\,W. M\"uller, A group-theoretical generalization of
  Pascal's triangle, \textit{Europ. J. Conbinatorics} \textbf{12}
  (1991), 43--49. 

\bibitem{MuDiss} T.\,W. M\"uller, Combinatorial aspects of finitely
generated virtually free groups, \textit{J. London Math. Soc.} (2)
\textbf{44} (1991), 75--94. 


\bibitem{MuHecke} T.\,W. M\"uller, Parity patterns in Hecke groups and
  Fermat  
primes, in: {\em Groups: Topological, Combinatorial and
    Arithmetic Aspects,} Proceedings of a conference held 1999 in 
  Bielefeld (T.\,W.~M\"uller, ed.), LMS Lecture Note Series, vol.~311,
  Cambridge University Press, Cambridge, 2004,   
pp.~327--374. 


\bibitem{Newman} M. Newman, Asymptotic formulas related to free
products of cyclic groups, \textit{Math. Comp.} \textbf{30} (1976),
838--846. 
 
\bibitem{Pride} S.\,J. Pride, The concept of largeness in group
theory. In: \textit{Word Problems II}, North Holland Publishing
Company, 1980, pp.~299--335. 

\bibitem{Schreier} O. Schreier, Die Untergruppen der freien Gruppen,
\textit{Abh. Math. Sem. Univ. Hamburg} \textbf{5} (1927), 161--183. 

\bibitem{Segal} D. Segal, Subgroups of finite index in soluble groups
I. In: \textit{Groups St Andrews 1985}, London Math. Soc.  Lecture
Note Series, vol.~121, Cambridge University Press, Cambridge, 1986,
pp.~307--314. 

\bibitem{Serre1} J.-P. Serre, Cohomologie des groupes discrets. In:
\textit{Prospects in Mathematics}, Ann. Math. Stud., vol.~70,
Princeton University Press, 1971, pp.~77--169. 

\bibitem{Serre2} J.-P. Serre, \textit{Arbres, Amalgames, $SL_2$},
Ast\'erisque, vol.~46, Soci\'et\'e math\'ematique de France, Paris,
1977. 

\bibitem{Specker} E. Specker, Die erste Cohomologiegruppe von
\"Uberlagerungen und Homotopieeigenschaften dreidimensionaler
Mannigfaltigkeiten, \textit{Comment. Math. Helv.} \textbf{23} (1949),
303--333.  

\bibitem{Stallings} J. Stallings, On torsion-free groups with
infinitely many ends, \textit{Ann. Math.} \textbf{88} (1968),
312--334. 

\bibitem{Wall} C.\,T.\,C. Wall, Poincar\'e complexes: I,
\textit{Ann. Math.} \textbf{86} (1967), 213--245. 

\end{thebibliography}
\end{document}